\documentclass[12pt,final]{article}
\usepackage{amsmath,amsthm,amsfonts,amssymb,graphicx,hyperref,enumerate,psfrag}
\usepackage{fullpage}

\usepackage[utf8]{inputenc} 

\newtheorem{lemma}{Lemma}

\newtheorem{remark}{Remark}
\newtheorem{proposition}{Proposition}
\newtheorem{theorem}{Theorem}
\newtheorem{conjecture}{Conjecture}

\newtheorem{corollary}{Corollary}
\newtheorem{question}{Question}

\newcommand{\EE}{{\mathbb{E}}}

\newcommand{\PP}{\mathbb{P}}

\newcommand{\R}{\mathbb {R}}

\newcommand{\cE}{\mathcal {E}}

\newcommand{\tmix}{\tau_{\textsc{mix}}}
\newcommand{\trel}{\tau_{\textsc{rel}}}
\newcommand{\tmls}{\tau_{\textsc{mls}}}
\newcommand{\tls}{\tau_{\textsc{ls}}}

\newcommand{\GG}{\mathbb {G}}

\newcommand{\am}{{\kappa_{\textsc{min}}}}
\newcommand{\ent}{{\mathrm{Ent}}}
\newcommand{\var}{{\mathrm{Var}}}

\newcommand{\xio}{\xi}
\title{A sharp log-Sobolev inequality for the multislice}
\author{Justin Salez\footnote{CEREMADE, CNRS, UMR 7534, Université Paris-Dauphine, PSL University, 75016 Paris, France}}

\date{}
\begin{document}
\maketitle

\begin{abstract}
We determine the log-Sobolev constant of the multi-urn Bernoulli-Laplace diffusion model with arbitrary parameters, up to a small universal multiplicative constant. Our result extends a classical estimate of Lee and Yau (1998) and confirms a conjecture of Filmus, O'Donnell and Wu (2018). Among other applications, we completely quantify the ``small-set expansion'' phenomenon on the multislice, and obtain sharp mixing-time estimates for the colored exclusion process on various graphs.
\end{abstract}
\tableofcontents

\section{Introduction}

\subsection{The multislice} Consider a sequence  of positive integers $\kappa=(\kappa_1,\ldots,\kappa_L)$ of arbitrary length $L\ge 2$, and set
\begin{eqnarray}
\label{sum}
n & = & \kappa_1+\cdots+\kappa_L.
\end{eqnarray}
We will refer to the elements of $[L]=\{1,\ldots,L\}$ as \emph{colors}, and write $\Omega_\kappa$ for the set of $[L]-$valued sequences in which each color $\ell\in[L]$ appears exactly $\kappa_\ell$ times:
\begin{eqnarray*}
\Omega_\kappa & := & \left\{\omega=(\omega_1,\ldots,\omega_n)\in[L]^n\colon \sum_{i=1}^n{\bf 1}_{(\omega_i=\ell)} = \kappa_\ell \textrm{ for each }\ell\in[L] \right\}.
\end{eqnarray*}
This natural combinatorial set is sometimes called a \emph{multislice}. It  provides a canonical interpretation for the classical  multinomial coefficient:
\begin{eqnarray*}
|\Omega_\kappa| & = & {n \choose \kappa_1,\ldots,\kappa_L}.
\end{eqnarray*}

The symmetric group $\mathfrak S_n$ acts transitively on the multislice in the obvious way, by permuting coordinates. In particular,  transpositions induce a natural local random walk on $\Omega_\kappa$,  which consists in repeatedly picking two positions $1\le i<j\le n$ uniformly at random and replacing the current state $\omega\in\Omega_\kappa$ with the new state
\begin{eqnarray*}
\omega^{ij} & := & \left(\omega_1,\ldots,\omega_{i-1},\omega_j,\omega_{i+1},\ldots,\omega_{j-1},\omega_i,\omega_{j+1},\ldots,\omega_n\right).
\end{eqnarray*}
This Markov chain is known as the \emph{transposition walk on the multislice}, or \emph{multi-urn Bernoulli-Laplace diffusion model} with parameter $\kappa$. It can also be viewed as a random walk on the Schreier graph $\GG_\kappa=(\Omega_\kappa,E_\kappa)$, whose  edge-set is given by
\begin{eqnarray*}
E_\kappa & := & \left\{\{\omega,\omega'\}\subseteq \Omega_\kappa\colon \sum_{i=1}^n{\bf 1}_{(\omega_i\ne\omega_i')}=2\right\}.
\end{eqnarray*}
Thanks to the degree of freedom in the choice of the parameter $\kappa$, the model is rich enough to encompass several classical special cases, including:
\begin{enumerate}[(i)]
\item the random walk on the complete graph of order $n$, corresponding to $\kappa=(1,n-1)$;
\item the $k-$particle Bernoulli-Laplace diffusion on $n$ sites, corresponding to $\kappa=(k,n-k)$;
\item the transposition walk on $\mathfrak S_n$, corresponding to $\kappa=(1,\ldots,1)$.
\end{enumerate}
These fundamental examples have been studied in full detail, see in particular \cite{MR626813,MR871832,MR2240787,MR2094147,MR958246,MR2166362,MR1675008,2019arXiv190508514T}. In the general case, however, understanding the precise impact of the parameter $\kappa$ on the mixing properties of the graph  $\GG_\kappa$  was suggested as an open problem several times \cite{MR871832, MR964069,MR3990024}. Beyond the traditional ``mixing times of Markov chains'' perspective, this question was recently shown in \cite{2018arXiv180903546F,MR3990024,MR4079633} to have remarkable  applications to the theory of Boolean functions on the multislice, see Section \ref{sec:appli} below for more details. In particular, the present paper was motivated by a conjecture from \cite{2018arXiv180903546F} regarding the so-called \emph{log-Sobolev constant} of the multislice, whose definition will be recalled in the next section.

\begin{remark}[Coarsening]\label{rk:coarsening}There is an obvious partial ordering on our parameter space: say that $\kappa'$ is coarser than $\kappa$ if it can be obtained from $\kappa$ by repeatedly merging two entries into one. Note that this operation simply  amounts to  identifying certain  colors, so that the transposition walk on $\Omega_{\kappa'}$ is a \emph{projection} of the one on $\Omega_{\kappa}$. In particular, the mixing behavior of the chain can only improve as $\kappa$ becomes coarser, with the case $\kappa=(1,\ldots,1)$ of example (iii) being the worst. Our main result will precisely quantify this qualitative statement. 
\end{remark}

\subsection{Functional inequalities} 
 One of the most powerful ways to quantify the mixing properties of a Markov chain consists in establishing appropriate functional inequalities for the underlying {Dirichlet form}. We shall here only recall the relevant definitions, and refer to the seminal papers \cite{MR1410112,MR2120475} or the excellent survey \cite{MR2341319} for a detailed account. 
 We start by turning the multislice $\Omega_\kappa$ into a probability space by equipping it with the uniform distribution. In particular, we regard functions  $f\colon\Omega_\kappa\to\R$ as random variables, and write $\EE_\kappa[f]$ for the corresponding expectation:
\begin{eqnarray*}
\EE_\kappa[f] & := & \frac{1}{|\Omega_\kappa|}\sum_{\omega\in\Omega_\kappa}f(\omega).
\end{eqnarray*}
 The \emph{Dirichlet form} of our chain is defined for every $f,g\colon\Omega_\kappa\to\R$ by 
\begin{eqnarray}
\label{def:dir}
\cE_\kappa\left(f,g\right) & := & \frac{1}{2n}\sum_{1\le i<j\le n}\EE_\kappa\left[\left(\nabla^{ij}f\right)\left(\nabla^{ij}g\right)\right],
\end{eqnarray}
where $(\nabla^{ij}f)(\omega):=f(\omega^{ij})-f(\omega)$ is the discrete gradient.

\begin{remark}[Scaling] We have here chosen to work under the natural continuous-time scaling where each of the $n\choose 2$ possible transpositions occurs at rate $1/n$, so that a coordinate gets refreshed at rate $1$. We emphasize that this is a matter of convention only: switching to discrete time amounts to nothing more that multiplying the above Dirichlet form by $2/(n-1)$. 
\end{remark}
Since $\cE_\kappa(f,f)$ measures the \emph{local} variation of the observable $f$ along a typical transition of the chain, it is natural to compare it with the \emph{variance} $\var_\kappa(f)$ or the \emph{entropy} $\ent_\kappa(f)$, which quantify the \emph{global} variation of $f$ across the whole state space:
\begin{eqnarray*}
\var_\kappa(f) & := & \EE_\kappa[f^2]-\EE^2_\kappa[f],\\
\ent_\kappa(f) & := & \EE_\kappa\left[f\log f\right]-\EE_\kappa[f]\log\EE_\kappa[f]. 
\end{eqnarray*}
All ${\rm logs}$ appearing in this paper are {natural} logarithms, and the last definition is of course restricted to non-negative functions, with the standard convention $0\log 0=0$. With this notation in hands, the three classical functional inequalities read  as follows:
\begin{itemize}
\item The \emph{Poincaré inequality} holds with constant $\tau$  if
\begin{eqnarray}
\label{PI}
{\var}_\kappa(f) & \le & \tau\, \cE_\kappa(f,f),\quad \textrm{for all }f\colon\Omega_\kappa\to\R.
\end{eqnarray}
\item The  \emph{modified log-Sobolev inequality} holds with constant $\tau$ if
\begin{eqnarray}
\label{MLSI}
\ent_\kappa(f) & \le & \tau\, \cE_\kappa\left(f,\log {f}\right),\quad \textrm{for all }f\colon\Omega_\kappa\to\R_+.
\end{eqnarray}
\item The \emph{log-Sobolev inequality}  holds with constant $\tau$ if
\begin{eqnarray}
  \label{LSI}
\ent_\kappa(f) & \le & \tau\, \cE_\kappa\left(\sqrt{f},\sqrt{f}\right),\quad \textrm{for all }f\colon\Omega_\kappa\to\R_+.
\end{eqnarray}
\end{itemize}

The optimal values of $\tau$ in these functional inequalities are respectively known as the (inverse)  \emph{Poincaré}, \emph{modified log-Sobolev}, and \emph{log-Sobolev} constants of the chain. They will here be denoted by $\trel(\kappa),\tmls(\kappa)$ and $\tls(\kappa)$. 
These fundamental parameters provide powerful controls on the underlying Markov semi-group,  and have tight connections to  mixing times, concentration of measure, small-set expansion, and hypercontractivity. We again refer to  \cite{MR1410112,MR2120475,MR2341319} for a detailed account, and to  \cite{MR3773799} for new characterizations. Let us simply note that the  statements (\ref{PI}), (\ref{MLSI}),  (\ref{LSI}) are essentially increasing in strength, in the sense that 
\begin{eqnarray}
\label{order}
2\trel(\kappa) \ \le \ 4\tmls(\kappa) \ \le \ \tls(\kappa).
\end{eqnarray}
Perhaps surprisingly, the first two quantities turn out to be too rough to capture the precise impact of $\kappa$ on the mixing properties of the multislice $\Omega_\kappa$. Specifically, we note the following dramatic insensitivity result, see Section \ref{sec:colored} for details. 
\begin{lemma}[Insensitivity of the Poincaré and modified log-Sobolev constants] \label{pr:insensitive} We have
 \begin{eqnarray*}
 \trel(\kappa) \ = \ 1 & \textrm{ and }
& \tmls(\kappa) \ \in \ \left[ \frac{1}{2},2\right],
 \end{eqnarray*}
 regardless of the choice of the parameter $\kappa$. 
\end{lemma}
In contrast, the much finer \emph{log-Sobolev constant} $\tls(\kappa)$ happens to depend on $\kappa$ in a non-trivial way, and understanding the exact nature of this dependency is precisely the aim of the present paper. Before we state our results, let us give a brief account on this general problem and its broad range of applications.

\subsection{Related works}

As already mentioned, the multi-urn Bernoulli-Laplace model encompasses various  well-studied special cases. The simplest one is the random walk on the complete $n-$vertex graph, obtained with $\kappa=(1,n-1)$. This example belongs to the short list of chains whose log-Sobolev constant is known exactly, see the seminal paper \cite{MR1410112} by Diaconis and Saloff-Coste. 
\begin{theorem}[Random walk on the complete graph, see Theorem A.1 in \cite{MR1410112}] \label{th:RW}
\begin{eqnarray*}
\tls(1,n-1) & = & \left\{
\begin{array}{ll}
\frac{n\log(n-1)}{n-2} & \textrm{if }n\ge 3\\
2 & \textrm{if }n=2.
\end{array}
\right.
\end{eqnarray*}
\end{theorem}

A much richer example is the famous ``Random Transposition'' walk on the symmetric group ${\mathfrak S}_n$, which corresponds to the choice $\kappa=(1,\ldots,1)$. A sharp estimate on  the log-Sobolev constant of this fundamental chain can be deduced from the detailed representation-theoretic analysis conducted by Diaconis and Shahshahani in their pioneering work \cite{MR626813}. 
\begin{theorem}[Random transposition on the symmetric group, see \cite{MR626813}]\label{th:RT}
\begin{eqnarray*}
\log n\ \le \ \tls(\underbrace{1,\ldots,1}_{n\textrm{ times}}) & \le & 4\log n.
\end{eqnarray*}
\end{theorem}
Several years later, Lee and Yau found a more direct proof, based on what is now known as the ``martingale method"  \cite{MR1675008}. This approach also allowed them to determine the order of magnitude of the log-Sobolev constant of the $k-$particle Bernoulli-Laplace diffusion on $n$ sites, thereby resolving an open problem raised by Diaconis and Saloff-Coste in \cite{MR1410112}. 
\begin{theorem}[Two-urn Bernoulli-Laplace diffusion model, see Theorem 5 in \cite{MR1675008}]\label{th:BL} There exists a universal constant $\varepsilon>0$ such that for all $0< k< n$,
 \begin{eqnarray*}
\varepsilon \log \left(\frac{n^2}{k(n-k)}\right)\ \le \ \tls\left(k,n-k\right) & \le & \frac{2}{\log 2}  \log \left(\frac{n^2}{k(n-k)}\right).
\end{eqnarray*}
\end{theorem}
The implications of Theorems  \ref{th:RT}-\ref{th:BL} are too numerous to be all cited. One particularly active direction consists in ``transferring'' these log-Sobolev estimates to models with less symmetry in order to obtain sharp mixing-time bounds, via the celebrated ``comparison method'' introduced by Diaconis and Saloff-Coste \cite{MR1245303,MR1233621}. Recent successful examples include the \emph{interchange process} on arbitrary graphs \cite{2018arXiv181110537A}, or the \emph{exclusion process} on high-dimensional product graphs \cite{2019arXiv190502146H}. Beyond Markov chains, the well-known connection between log-Sobolev inequalities and \emph{hypercontractivity} provides another extremely fertile ground for applications  in discrete analysis and computer science. We refer to the book  \cite[Chapters 9 \& 10]{MR3443800} for details, and  to the recent work \cite{2018arXiv180903546F} for an impressive  list of references from combinatorics, computational learning, property testing or Boolean functions, where Theorems  \ref{th:RT}-\ref{th:BL} played a crucial role. Motivated by these applications, {Filmus}, {O'Donnell} and {Wu} \cite{2018arXiv180903546F} initiated the  study of the log-Sobolev constant $\tls(\kappa)$ for general $\kappa$. Their main result is as follows. 
\begin{theorem}[General bound, see Theorem 1 in \cite{2018arXiv180903546F}]\label{th:F}For any choice of the parameter $\kappa$,
 \begin{eqnarray*}
\tls\left(\kappa\right) & \le & \frac{2}{\log 2} \sum_{\ell=1}^L \log \left(\frac{4n}{\kappa_\ell}\right).
\end{eqnarray*}
\end{theorem}
Several remarkable consequences of this estimate can be found in the recent works  \cite{2018arXiv180903546F,MR4079633}. A quick comparison with Theorems \ref{th:RW},  \ref{th:RT} and \ref{th:BL} shows that the bound is of the right order of magnitude in the extreme case $L=2$, but is off by a factor of order $n$ at the other extreme, $L=n$. Regarding what the correct order of magnitude of $\tls(\kappa)$ should be for \emph{all} ranges of $\kappa$, {Filmus}, {O'Donnell} and {Wu} proposed the following beautifully simple dependency. 
\begin{conjecture}[See page 3 in \cite{2018arXiv180903546F}]\label{cj:main}For any choice of the parameter $\kappa$,
\begin{eqnarray*}
 \tls(\kappa) & \asymp & \log \left(\frac{n}{\am}\right),
\end{eqnarray*}
where $\am:=\min\{\kappa_1,\ldots,\kappa_L\}$ and where $\asymp$ means equality up to universal pre-factors. 
\end{conjecture}
Note that the right-hand side decreases smoothly from $\log n $ downto $0$ as $\kappa$ becomes coarser and coarser, in agreement with Remark \ref{rk:coarsening}.  
To better appreciate this conjecture,  consider the \emph{single-site dynamics} obtained by projecting the multislice onto a fixed coordinate $i\in[n]$: under our transposition walk, the variable  $\omega_i$  simply gets  refreshed at unit rate according to the marginal distribution
\begin{eqnarray*}
\PP_\kappa\left(\omega_i=\ell\right) & = & \frac{\kappa_\ell}{n},\qquad \ell\in[L].
\end{eqnarray*}
The log-Sobolev constant of this trivial chain is well-known to be 
\begin{eqnarray*}
\tls^{\textrm{triv}}(\kappa) & = & \frac{n}{n-{2\am}}\log\left(\frac{n}{\am}-1\right) \ \asymp \ \log \left(\frac{n}{\am}\right),
\end{eqnarray*}
see  \cite[Theorem A.1]{MR1410112}.
Although our probability space $\Omega_\kappa$ is far from being a product space,  the above conjecture asserts that the transposition walk  mixes essentially as well as if the coordinates $\omega_1,\ldots,\omega_n$ were being refreshed independently.  A brief look at Theorems \ref{th:RW}, \ref{th:RT} and \ref{th:BL} will convince the reader that this intuition is correct in all known special cases. 

\section{Results}

 \subsection{Main estimate}
 \label{sec:appli}
 Our main result is the determination of the log-Sobolev constant $\tls(\kappa)$ for all values of the parameter $\kappa$, up to a (small) universal multiplicative constant.
\begin{theorem}[The log-Sobolev constant of the multislice]\label{th:main}For all values of  $\kappa$,
\begin{eqnarray*}
\log\left(\frac{n}{\am}\right) \ \le \ \tls(\kappa)  & \le & \frac{4}{\log 2}\log\left(\frac{n}{\am}\right).
\end{eqnarray*}
\end{theorem}
This confirms Conjecture \ref{cj:main}. We note that the improvement upon  Theorem \ref{th:F} can be considerable  if  the dimension $L$ is large. Specifically, the upper bound of Filmus, {O'Donnell} and {Wu} is \emph{always} super-linear in $L$, since the convexity of $t\mapsto t\log t$ yields
\begin{eqnarray*}
\sum_{\ell=1}^L \log \left(\frac{n}{\kappa_\ell}\right) & \ge & L\log L,
\end{eqnarray*}
for any choice of the parameter $\kappa$. In contrast, our result shows that
\begin{eqnarray}
\label{llogl}
\tls(\kappa) & \asymp & \log L,
\end{eqnarray}
 as long as the vector $\kappa=(\kappa_1,\ldots,\kappa_L)$ is reasonably \emph{balanced}, in the (weak) sense that its lowest entry is of the same order as the mean entry. In particular, our estimate can be readily used to sharpen the dependency in $L$ in the various quantitative results that were derived from  Theorem \ref{th:F} in \cite{2018arXiv180903546F}. To avoid a lengthy detour through hypercontractivity, we choose to leave the details to the reader, and to instead describe two different applications: a sharp quantification of the ``small-set expansion'' phenomenon for the multislice, and a general log-Sobolev inequality for the {colored exclusion processes}.

\begin{remark}[Sharpness of universal constants]\label{rk:sharp}In our lower bound, the pre-factor in front of the logarithm can not be replaced by any larger universal constant, since we have 
\begin{eqnarray*}
\tls(\kappa) & = & \left(1+o(1)\right)\log\left(\frac{n}{\am}\right),
\end{eqnarray*}
in the special case $\kappa=(1,n-1)$, as per Theorem \ref{th:RW}. Regarding the upper bound, our pre-factor can not be improved by more than a $\log 2$ factor. Indeed, we will show that
\begin{eqnarray*}
\tls(\kappa) & \ge & \left(4-o(1)\right)\log\left(\frac{n}{\am}\right),
\end{eqnarray*}
in the important special case $\kappa=\left(\lfloor n/2\rfloor,\lceil n/2\rceil\right)$, see (\ref{sharpiota}). In fact, the possibly loose $\log 2$ term comes directly from the one appearing in Theorem \ref{th:BL}, and any improvement of the latter will immediately imply the same improvement in our upper bound.
\end{remark}
\subsection{Small-set expansion} 
 Recall that the multislice is naturally equipped with a graph structure by declaring two vertices $\omega,\omega'\in\Omega_\kappa$ to be adjacent if they differ at exactly two coordinates. Following standard graph-theoretical notation, we write $|\partial A|$ for the \emph{edge boundary} of a subset $A\subseteq \Omega_\kappa$, i.e., the set of edges  having one end-point in $A$ and the other outside $A$. Let us consider the problem of finding a constant $\iota(\kappa)$, as large as possible, such that  the isoperimetric inequality
\begin{eqnarray}
\label{isop}
\frac{|\partial A|}{|A|} & \ge & \iota(\kappa)\,\log\left(\frac{|\Omega_\kappa|}{|A|}\right),
\end{eqnarray}
holds for all non-empty subsets $A\subseteq \Omega_\kappa$. The left-hand side measures the \emph{conductance} of $A$, i.e. the facility for the walk to escape from $A$, given that it currently lies in $A$. The presence of the logarithmic term on the other side constitutes a notable improvement upon the more standard \emph{Cheeger inequality}: instead of being constant, the right-hand side of (\ref{isop}) gets larger as the set $A$ gets smaller, thereby capturing the celebrated \emph{small-set expansion} phenomenon  \cite{21923,MR1798047, 2018arXiv180903546F}. Our log-Sobolev estimate allows us to  determine the fundamental quantity $\iota(\kappa)$ for all values of $\kappa$, up to a small universal constant. 
\begin{corollary}[Small-set expansion for the multislice]\label{co:exp}The optimal constant in (\ref{isop}) satisfies
\begin{eqnarray}
\frac{\log 2}{4}\frac{n}{\log\left(\frac{n}{\am}\right)} \ \le \ \iota(\kappa) & \le &  \frac{n}{\log\left(\frac{n}{\am}\right)}.
\end{eqnarray}
\end{corollary}
The proof will be given in Section \ref{sec:final}. As in Remark \ref{rk:sharp},  the universal constants appearing in our estimate can not be improved, apart from perhaps removing the $\log 2$ term. 

\subsection{Colored exclusion process} A far-reaching generalization of the transposition walk on the multislice $\Omega_\kappa$ consists in  allowing each of the $n\choose 2$ possible transpositions to occur at a different (possibly zero) rate. More precisely, we fix a non-negative symmetric array  $G=(G_{ij})_{1\le i,j\le n}$ (which we interpret as a weighted graph) and consider the following weighted version of the Dirichlet form (\ref{def:dir}):
\begin{eqnarray}
\label{def:weighted}
\cE_{\kappa}^G\left(f,g\right) & := & \frac{1}{2}\sum_{1\le i<j\le n}G_{ij}\,\EE_\kappa\left[\left(\nabla^{ij}f\right)\left(\nabla^{ij}g\right)\right].
\end{eqnarray}
The canonical setting -- to which we shall here stick for simplicity -- consists in taking $G$ to be the transition matrix of the simple random walk on a regular graph, which we henceforth identify with $G$. The resulting process is known as  the $\kappa-$\emph{colored exclusion process} on $G$, see \cite{MR2629990}. By varying the parameter $\kappa$, we obtain a rich family of diffusion models on $G$ including:
\begin{enumerate}[(i)]
\item the \emph{simple random walk} on $G$, when $\kappa=(1,n-1)$;
\item the \emph{$k-$particle exclusion process} on $G$, when $\kappa=(k,n-k)$;
\item the \emph{interchange process} on $G$, when $\kappa=(1,\ldots,1)$.
\end{enumerate}
Comparing the mixing properties of these three processes constitutes a rich and active research problem, see \cite{MR2023023, MR2244427, MR2629990, MR3069380,MR3077529, MR3978223, hermon2018exclusion,2018arXiv181110537A}. Perhaps the most celebrated result in this direction is the remarkable fact that their Poincaré constants coincide, as conjectured by Aldous and established by Caputo, Liggett and Richthammer \cite{MR2629990}. 
 \begin{theorem}[Insensitivity of the Poincaré constant, see \cite{MR2629990}]\label{th:aldous}
 The Poincaré constant $\trel(\kappa,G)$ of the $\kappa-$colored exclusion process on $G$ does not depend on $\kappa$. In particular, it equals the Poincaré constant $\trel(G)$ of the simple random walk on $G$. 
 \end{theorem}
 In a sense, this result asserts that the Poincaré constant  is too ``rough'' to capture the influence of the color profile $\kappa$ on the mixing properties of the colored exclusion process. It is thus natural to turn one's attention to  the finer log-Sobolev constant.

 \begin{question}
 How does the log-Sobolev constant $\tls(\kappa,G)$ depend upon the parameter $\kappa$ ? 
 \end{question} Our main result answers this question in the simple \emph{mean-field} setting, where $G$ is the complete graph. However, it implies an estimate of  $\tls(\kappa,G)$ for arbitrary $G$, by means of the celebrated ``comparison method'' introduced by  Diaconis and Saloff-Coste \cite{MR1245303,MR1233621}. A particularly pleasant observation here is that we do not even need to build a comparison theory for the colored exclusion process: we can simply recycle the one that has already been developed for the interchange process. Specifically, let $c(G)$ be the smallest number such that the functional inequality
\begin{eqnarray}
\label{comp}
\cE^G_{(1,\ldots,1)}(f,f) & \le & c(G)\,\cE_{(1,\ldots,1)}(f,f),
\end{eqnarray}
holds for all $f\colon\Omega_{(1,\ldots,1)}\to\R$.  This fundamental quantity is known as the \emph{comparison constant} of the interchange process on $G$. It was shown in \cite{2018arXiv181110537A} that 
\begin{eqnarray*}
c(G) & \lesssim & \tmix(G),
\end{eqnarray*}
where $\lesssim$ means inequality up to a universal multiplicative constant,  and where $\tmix(G)$ denotes the mixing time of the simple random walk on $G$. It is in fact believed that 
\begin{eqnarray}
\label{conj}
c(G) & \asymp & \trel(G),
\end{eqnarray}
see Conjecture 2 in \cite{hermon2019interchange}. This refinement, inspired by an analogous relation for the \emph{Zero-Range process} \cite{MR3984254}, is already known to hold for several natural families of graphs ranging from low-dimensional tori \cite{2018arXiv181110537A} to high-dimensional products \cite{hermon2019interchange}.  Those estimates can be combined with our main result to yield a general log-Sobolev inequality for the colored exclusion process (see Section \ref{sec:colored} for details): 
\begin{corollary}[Log-Sobolev inequality for the colored exclusion process]\label{co:colored}We have
\begin{eqnarray*}
\max\left\{2\trel(G),\log\left(\frac{n}{\am}\right)\right\} \ \le \ \tls(\kappa, G) & \le & \frac{4}{\log 2}\,c(G)\log\left(\frac{n}{\am}\right).
\end{eqnarray*}
\end{corollary}
To appreciate the sharpness of this general inequality, note that the lower and upper bounds are of the same order in the following two generic situations:
\begin{itemize}
\item For families of graphs with $c(G)\asymp 1$ (i.e. ``well-connected'' graphs), we obtain
\begin{eqnarray*}
\tls(\kappa, G) & \asymp & \log\left(\frac{n}{\am}\right),
\end{eqnarray*}
exactly as in the mean-field case. Note that this potentially constitutes a considerable extension of our main result, since the class of graphs satisfying  $c(G)\asymp 1$ is believed to contain all expanders, as per (\ref{conj}). 
\item For graphs satisfying the conjecture (\ref{conj}), in the regime $\am\ge \varepsilon n$ ($\varepsilon>0$ fixed), we get
\begin{eqnarray*}
\tls(\kappa, G) & \asymp & \trel(G).
\end{eqnarray*}
This constitutes a multi-colored generalization of several estimates obtained for two colors, including  \cite[Theorem 4]{MR1675008} on the cycle and \cite[Corollary 5]{hermon2019interchange} on the hypercube. 
\end{itemize}
\begin{remark}[Mixing times]
One of the many interests of those log-Sobolev estimates is that they provide powerful controls on the strong $L^\infty-$mixing time of the process, see e.g., \cite{MR2341319}. Let us here just give one concrete example: on the $d-$dimensional hypercube, our work implies that the balanced colored exclusion process with an arbitrarily fixed number $L\ge 2$ of colors mixes in time $\Theta(d^2)$. The special case $L=2$ of this statement had been conjectured several years ago by Wilson \cite{MR2023023}, and was settled only recently \cite{hermon2018exclusion}. 
\end{remark}
We end this section with an intriguing possibility, which arises naturally  in view of Theorem \ref{th:aldous} and of what happens in the mean-field case (Lemma \ref{pr:insensitive}).
\begin{question}[Sensitivity of the modified log-Sobolev constant]Can the choice of the parameter $\kappa$ affect  $\tmls(\kappa,G)$ by more than a universal multiplicative constant ?
\end{question}
 A negative answer would, in particular, substantially improve our current knowledge on the mixing times of the interchange and exclusion processes on general graphs. 
We note that, unlike our main result, the estimate on $\tmls(\kappa)$ provided by Lemma \ref{pr:insensitive} can \emph{not} be directly transferred to more general graphs, since the modified log-Sobolev constant is notoriously \emph{not} amenable to comparison techniques. This severe drawback constitutes a strong point in favor of log-Sobolev inequalities (as opposed to their modified versions) for mean-field interacting particle models, and was one of the motivations for the present work.

\section{Proofs}

\subsection{General strategy} 
\label{sec:strategy}
Let us start with an elementary but crucial observation about the multislice.
\begin{remark}[Recursive structure]\label{obs} If $(\omega_1,\ldots,\omega_n)$ is uniformly distributed on  $\Omega_\kappa$, then the conditional law of $(\omega_1,\ldots,\omega_{i-1},\omega_{i+1},\ldots,\omega_n)$ given $\{\omega_i=\ell\}$ is uniform on $\Omega_{\kappa'}$, where 
$$\kappa'=\left(\kappa_1,\ldots,\kappa_{\ell-1},\kappa_\ell-1,\kappa_{\ell+1},\ldots,\kappa_L\right).$$
\end{remark}
Such a simple recursive structure suggests the possibility of proving Theorem \ref{th:main}  by induction over the dimension $n$, using the ``chain rule'' for entropy (see formula (\ref{decomposition}) below).  This is in fact a classical  strategy for establishing  functional inequalities, known as the ``martingale method''. Introduced  by Lu \& Yau \cite{MR1233852} in the context of Kawasaki and Glauber dynamics, it has been successfully applied to various  interacting particle systems \cite{MR1483598,MR1675008,MR2023890,MR2094147,Cap,hermon2019entropy}, as well as other Markov chains enjoying an appropriate recursive structure \cite{MR1944012,1181997,MR2099650,2018arXiv180903546F,2019arXiv190202775H}. In particular, this is how Theorem \ref{th:BL}  was proved. However, as explained in detail in \cite{2018arXiv180903546F}, moving from the special case $L=2$ covered by Theorem \ref{th:BL} to the general case studied in Theorem \ref{th:F} significantly complicates the inductive argument, resulting in the loose $L\log L$ dependency mentioned at (\ref{llogl}). Here we introduce two simple ideas to bypass those complications and prove Conjecture \ref{cj:main}:
\begin{enumerate}[(i)]
\item instead of just a single site, we condition on a whole region being colored with $\ell\in[L]$;
\item when averaging the contributions from the various colors, we assign more weight to rare colors, which are the one which really govern $\tls(\kappa)$. More precisely, our decomposition (\ref{dec:ent}) below gives weight $1-\frac{\kappa_\ell} n$ to the $\ell-$colored region, whereas the traditional uniform average over all sites would give it the weight $\frac{\kappa_\ell}{n}$.  
\end{enumerate}
Let us now implement those ideas. We fix an observable $f\colon\Omega_\kappa\to\R_+$ once and for all. To lighten notation, we drop the index $\kappa$ from our expectations, and write simply
\begin{eqnarray*}
 \ent(f) & := & \EE[f\log f] -\EE[f]\log\EE[f],
  \end{eqnarray*}  
 for the entropy of $f$. If $Z$ is a random variable on $\Omega_\kappa$, we define the \emph{conditional entropy} of $f$ given  $Z$  by simply replacing all expectations with conditional expectations, i.e.
 \begin{eqnarray*}
 \ent(f|Z) & := & \EE[f\log f|Z] -\EE[f|Z]\log\EE[f|Z].
  \end{eqnarray*} 
We then have the following elementary ``chain rule'':
\begin{eqnarray}
\label{decomposition}
\ent(f) & = & \EE\left[\ent(f|Z)\right]+\ent\left(\EE[f|Z]\right).
\end{eqnarray}
The choice $Z=\omega_i$ is of course natural in light of Remark \ref{obs}, and this was the one adopted in the proofs of Theorems \ref{th:BL} and \ref{th:F}. However, as mentioned in (i) above, we choose here to condition instead on the whole $\ell-$colored region, i.e., on the random set
\begin{eqnarray}
\label{def:xio}
\xio_\ell & := & \left\{i\in[n]\colon \omega_i=\ell\right\}.
\end{eqnarray}
With $Z=\xio_\ell$, the formula (\ref{decomposition}) becomes
\begin{eqnarray}
\label{dec:set}
\ent(f) & = & \EE\left[\ent\left(f|\xio_\ell\right)\right]+\ent\left(\EE\left[f|\xio_\ell\right]\right).
\end{eqnarray}
Following our second idea (ii), we multiply both sides of this identity by the ``unusual'' weight $1-\frac{\kappa_\ell}{n}$ and then sum over all colors $\ell\in[L]$. Recalling  (\ref{sum}), we obtain the following formula, which will constitute the basis of our induction:
\begin{eqnarray}
\label{dec:ent}
(L-1)\ent(f) & = & \underbrace{\sum_{\ell=1}^L\left(1-\frac{\kappa_\ell}{n}\right)\EE\left[\ent\left(f|\xio_\ell\right)\right]}_{\Sigma_1}+\underbrace{\sum_{\ell=1}^L\left(1-\frac{\kappa_\ell}{n}\right)\ent\left(\EE\left[f|\xio_\ell\right]\right)}_{\Sigma_2}.
\end{eqnarray}
Our main task will consist in estimating the two terms $\Sigma_1$ and $\Sigma_2$ on the right-hand side, in terms of the log-Sobolev constants of certain lower-dimensional multislices. More precisely, we let $\kappa^{\setminus \ell}$ denote the parameter obtained from $\kappa$ by removing the $\ell-$th entry, i.e.
\begin{eqnarray*}
\kappa^{\setminus\ell} &:= & \left(\kappa_1,\ldots,\kappa_{\ell-1},\kappa_{\ell+1},\ldots,\kappa_L\right),
\end{eqnarray*}
and we will prove in the next section that
\begin{eqnarray}
\label{pr:first}
\Sigma_1 & \le & (L-2)\max_{\ell\in[L]}\left\{\tls\left(\kappa^{\setminus \ell}\right)\right\}\cE_\kappa\left(\sqrt{f},\sqrt{f}\right);\\
\label{pr:second}
\Sigma_2 & \le & \max_{\ell\in[L]}\left\{2\left(1-\frac{\kappa_\ell}n\right)\tls\left(\kappa_\ell,n-\kappa_\ell\right)\right\}\cE_\kappa\left(\sqrt{f},\sqrt{f}\right).
\end{eqnarray}
Plugging those estimates into (\ref{dec:ent}) yields a log-Sobolev inequality for $\Omega_\kappa$, thereby establishing the following recursive estimate.
\begin{proposition}[Recursive log-Sobolev estimate]\label{pr:recursive}We have
\begin{eqnarray*}
(L-1)\tls\left(\kappa\right) & \le & (L-2)\max_{\ell\in[L]}\left\{\tls\left(\kappa^{\setminus \ell}\right)\right\}+\max_{\ell\in[L]}\left\{2\left(1-\frac{\kappa_\ell}n\right)\tls\left(\kappa_\ell,n-\kappa_\ell\right)\right\}.
\end{eqnarray*}
\end{proposition}
From this, the upper bound in Theorem \ref{th:main} follows by an easy induction over the number $L$ of colors, using the known log-Sobolev estimate for $L=2$ (Theorem \ref{th:RT}). The details, as well as the proof of the lower bound, are provided in Section \ref{sec:final}.

\subsection{Main recursion}
This section is devoted to proving the two technical estimates (\ref{pr:first}) and (\ref{pr:second}) which, in view of the decomposition (\ref{dec:ent}), establish  Proposition \ref{pr:recursive}. 
\begin{proof}[Proof of the first estimate (\ref{pr:first})]
Conditionally on the $\ell-$colored region $\xio_\ell$,  $f$ may be regarded as a function of the remaining coordinates $(\omega_i\colon i\in[n]\setminus \xio_\ell)$, which form a uniformly distributed element of  $\Omega_{\kappa^{\setminus\ell}}$. Consequently, the log-Sobolev inequality for the multislice $\Omega_{\kappa^{\setminus\ell}}$ gives
\begin{eqnarray*}
\ent\left(f|\xio_\ell\right) & \le &  \frac{\tls(\kappa^{\setminus\ell})}{2(n-\kappa_\ell)}\sum_{1\le i<j\le n}\EE\left[\left.\left(\nabla^{ij}\sqrt{f}\right)^2{\bf 1}_{(i\notin\xi_\ell,j\notin\xi_\ell)}\right|\xio_\ell\right].
\end{eqnarray*}
Note that the event in the indicator can be rewritten as $\{\ell\notin\{\omega_i,\omega_j\}\}$, and that we may impose the restriction $\{\omega_i\ne\omega_j\}$ at no cost, since $\nabla^{ij}\sqrt{f}=0$ on the event $\{\omega_i=\omega_j\}$.
Taking expectations and rearranging, we arrive at
\begin{eqnarray*}
\left(1-\frac{\kappa_\ell}{n}\right)\EE\left[\ent\left(f|\xio_\ell\right)\right]  & \le & \frac{\tls(\kappa^{\setminus\ell})}{2n}\sum_{1\le i<j\le n}\EE\left[\left(\nabla^{ij}\sqrt{f}\right)^2{\bf 1}_{(\omega_i\ne\omega_j)}{\bf 1}_{(\ell\notin\{\omega_i,\omega_j\})}\right].
\end{eqnarray*}
Summing over all $\ell\in[L]$ yields
\begin{eqnarray*}
\sum_{\ell=1}^L\left(1-\frac{\kappa_\ell}{n}\right)\EE\left[\ent\left(f|\xio_\ell\right)\right] & \le & (L-2)\max_{\ell\in[L]}\left\{\tls\left(\kappa^{\setminus \ell}\right)\right\}\cE_\kappa\left(\sqrt{f},\sqrt{f}\right),
\end{eqnarray*}
which is exactly the claim made at (\ref{pr:first}).
\end{proof}

\begin{proof}[Proof of the second estimate (\ref{pr:second})]
Fix $\ell\in[L]$, and let us write
\begin{eqnarray}
\label{def:F}
\EE[f|\xio_\ell] & = & F(\xio_\ell),
\end{eqnarray}
for some non-negative function $F=F_\ell$. The distribution of $\xio_\ell$ is uniform  over all $\kappa_\ell-$element subsets of $[n]$, and this is precisely the stationary distribution of the occupied set in the $\kappa_\ell-$particle Bernoulli-Laplace diffusion model on $n$ sites. When applied to the function $F$, the log-Sobolev inequality for this process reads as follows:
\begin{eqnarray}
\label{induced}
\ent\left(\EE[f|\xio_\ell]\right) & \le & \frac {\tls(\kappa_\ell,n-\kappa_\ell)}{2n}\sum_{1\le i<j\le n}\EE\left[\left(\sqrt{F\left(\xio_\ell^{ij}\right)}-\sqrt{F\left(\xio_\ell\right)}\right)^2\right],
\end{eqnarray}
where $A^{ij}$ denotes the set obtained from $A$ by swapping the membership status of $i$ and $j$:
\begin{eqnarray*}
A^{ij} & := & \left\{
\begin{array}{ll}
A\cup\{j\}\setminus\{i\} & \textrm{if }i\in A,j\notin A\\
A\cup\{i\}\setminus\{j\} & \textrm{if }i\notin A,j\in A\\
A & \textrm{ else.}
\end{array}
\right.
\end{eqnarray*}
Now, fix $1\le i<j\le n$ and a $\kappa_\ell-$element set $A\subseteq[n]$. First, by definition of $F$, we have
\begin{eqnarray*}
\EE[f|\xio_\ell=A] & = & F(A).
\end{eqnarray*}
On the other hand, since the involution $\tau^{ij}\colon\omega\mapsto \omega^{ij}$ preserves the uniform law on $\Omega_\kappa$ and maps the event $\{\xi_\ell=A\}$ onto the event $\{\xi_\ell=A^{ij}\}$, we have 
\begin{eqnarray*}
\EE\left[f\circ\tau^{ij}|\xio_\ell=A\right] & = & \EE\left[f|\xio_\ell=A^{ij}\right] \ = \ F\left(A^{ij}\right).
\end{eqnarray*}
But the function
$
\Phi\colon (u,v)\mapsto (\sqrt{u}-\sqrt{v})^2
$
is convex on $\R_+^2$, so Jensen's inequality yields
\begin{eqnarray*}
\left(\sqrt{F(A^{ij})}-\sqrt{F(A)}\right)^2
& = & \Phi\left(\EE\left[f\circ\tau^{ij}|\xio_\ell=A\right],\EE\left[f|\xio_\ell=A\right]\right)\\
& \le & \EE\left[\Phi(f\circ\tau^{ij},f)|\xio_\ell=A\right] \ = \ \EE\left[\left.\left(\nabla^{ij} \sqrt{f}\right)^2\right|\xio_\ell=A\right].
\end{eqnarray*}
Moreover, we have $\left(\sqrt{F(A^{ij})}-\sqrt{F(A)}\right)^2=0$ when $A$ contains neither $i$ nor $j$, so we obtain
\begin{eqnarray*}
\left(\sqrt{F(A^{ij})}-\sqrt{F(A)}\right)^2 & \le & \EE\left[\left.\left(\nabla^{ij} \sqrt{f}\right)^2\right|\xio_\ell=A\right]\left({\bf 1}_{(i\in A)}+{\bf 1}_{(j\in A)}\right).
\end{eqnarray*}
Averaging this inequality over all possible $\kappa_\ell-$element set $A\subseteq[n]$ yields
\begin{eqnarray*}
\EE\left[\left(\sqrt{F\left(\xio_\ell^{ij}\right)}-\sqrt{F(\xio_\ell)}\right)^2\right] & \le & \EE\left[\left(\nabla^{ij} \sqrt{f}\right)^2\left({\bf 1}_{(i\in \xio_\ell)}+{\bf 1}_{(j\in \xio_\ell)}\right)\right].
\end{eqnarray*}
We may now plug this estimate back into (\ref{induced}) to arrive at 
\begin{eqnarray*}
\ent\left(\EE[f|\xio_\ell]\right) & \le & \frac{\tls(\kappa_\ell,n-\kappa_\ell)}{2n}\sum_{1\le i<j\le n}\EE\left[\left(\nabla^{ij} \sqrt{f}\right)^2\left({\bf 1}_{\left(\omega_i=\ell\right)}+{\bf 1}_{\left(\omega_j=\ell\right)}\right)\right].
\end{eqnarray*}
Finally, multiplying by $\left(1-\frac{\kappa_\ell}{n}\right)$ and summing over all $\ell\in[L]$ gives
\begin{eqnarray*}
\sum_{\ell=1}^L\left(1-\frac{\kappa_\ell}{n}\right)\ent\left(\EE\left[f|\xio_\ell\right]\right) & \le & \max_{\ell\in[L]}\left\{2\left(1-\frac{\kappa_\ell}n\right)\tls\left(\kappa_\ell,n-\kappa_\ell\right)\right\}\cE_\kappa\left(\sqrt{f},\sqrt{f}\right),
\end{eqnarray*}
which is precisely the claim (\ref{pr:second}). 
\end{proof}

\subsection{Putting things together}
\label{sec:final}

To complete the proof of Theorem \ref{th:main}, we only need an estimate on the second term appearing on the right-hand side of our recursive log-Sobolev inequality. We of course use Theorem \ref{th:BL}.
\begin{lemma}[Two-color estimate]\label{lm:L2}For any $\ell\in[L]$, we have
\begin{eqnarray*}
\left(1-\frac{\kappa_\ell}{n}\right)\tls\left(\kappa_\ell,n-\kappa_\ell\right) & \le & \frac{2}{\log 2}\log\left(\frac{n}{\am}\right).
\end{eqnarray*}
\end{lemma}
\begin{proof}By Theorem \ref{th:BL}, we have
\begin{eqnarray*}
\left(1-\frac{\kappa_\ell}{n}\right)\tls\left(\kappa_\ell,n-\kappa_\ell\right)  & \le & \frac{2}{\log 2}\left(1-\frac{\kappa_\ell}{n}\right)\log\left(\frac{n^2}{\kappa_\ell(n-\kappa_\ell)}\right).
\end{eqnarray*}
Since the right-hand side is maximized when $\kappa_\ell=\am$, our task boils down to establishing
\begin{eqnarray*}
\left(1-\frac{\am}{n}\right)\log\left(\frac{n^2}{\am(n-\am)}\right) & \le & \log\left(\frac{n}{\am}\right).
\end{eqnarray*}
But this is exactly the special case $t=\frac{\am}{n}$ of the inequality
\begin{eqnarray*}
t\log t-(1-t)\log\left(1-t\right) & \le & 0,
\end{eqnarray*}
which is valid for all $t\in[0,\frac 12]$. To see this, note that  the left-hand side is a convex function of $t\in[0,\frac 12]$ (as can be easily checked by differentiating) and that it equals zero at the two boundary points $t=0$ and $t=\frac 12$. 
\end{proof}
We are now in position to prove our main result. 
\begin{proof}[Proof of the upper bound in Theorem \ref{th:main}]Our aim is to prove that
\begin{eqnarray}
\label{claim}
\tls(\kappa) & \le & \Phi(\kappa) \ := \ \frac{4}{\log 2}\log\left(\frac{n}{\am}\right).
\end{eqnarray}
We proceed by induction over the dimension $L$ of the parameter $\kappa=(\kappa_1,\ldots,\kappa_L)$. 
By combining Proposition \ref{pr:recursive} and Lemma \ref{lm:L2}, we have
\begin{eqnarray}
\label{end}
(L-1)\tls(\kappa) & \le & \Phi(\kappa)+(L-2)\max_{\ell\in[L]}\tls(\kappa^{\setminus\ell}),
\end{eqnarray}
which already establishes the claim in the base case $L=2$. Now, assume that $L\ge 3$ and that the claim already holds for lower values of $L$. In particular, we know that
\begin{eqnarray*}
\tls(\kappa^{\setminus \ell}) & \le & \Phi(\kappa^{\setminus \ell}),
\end{eqnarray*}
for all $\ell\in[L]$. But $\Phi(\kappa^{\setminus\ell})\le \Phi(\kappa)$, since removing an entry from the parameter $\kappa$ can only decrease the value of the sum $n=\kappa_1+\cdots+\kappa_L$ and increase the value of the minimum $\am=\min\{\kappa_1,\ldots,\kappa_L\}$. Consequently, (\ref{end}) gives
\begin{eqnarray*}
(L-1)\tls(\kappa) & \le & \Phi(\kappa)+(L-2)\Phi(\kappa)\ =\ (L-1)\Phi(\kappa), 
\end{eqnarray*}
and (\ref{claim}) is established.
\end{proof}
Our upper bound on $\tls(\kappa)$ implies the lower bound on $\iota(\kappa)$ given in Corollary \ref{co:exp}, thanks to the well-known relation between log-Sobolev inequalities and small-set expansion:
\begin{lemma}[Log-Sobolev inequality and small-set expansion]\label{lm:vs}We have 
\begin{eqnarray*}
\iota(\kappa)\tls(\kappa) & \ge & n.
\end{eqnarray*}
\end{lemma}
\begin{proof}This follows from the definitions of $\iota(\kappa)$ and $\tls(\kappa)$, once we have observed that 
\begin{eqnarray*}
\cE_\kappa\left({\bf 1}_A,{\bf 1}_A\right) & = & \frac{|\partial A|}{n|\Omega_\kappa|},\\
\ent_\kappa\left({\bf 1}_A\right) & = & \frac{|A|}{|\Omega_\kappa|}\log\left(\frac{|\Omega_\kappa|}{|A|}\right),
\end{eqnarray*}
for any event $A\subseteq \Omega_\kappa$. 
\end{proof}
The inequality in Lemma \ref{lm:vs} is obtained by restricting the definition of the log-Sobolev inequality to indicator functions, and could therefore be rather loose. However, it turns out to be sharp in the present case, as we will now see.
\begin{proof}[Proof of the remaining halves of Theorem \ref{th:main} and  Corollary \ref{co:exp}]By definition, we have 
\begin{eqnarray}
\label{iota}
\iota(\kappa) & \le  & \frac{|\partial A|}{|A|\log\left(\frac{|\Omega_\kappa|}{|A|}\right)},
\end{eqnarray}
for any non-empty event $A\subseteq \Omega_\kappa$. We fix $\ell\in[L]$ such that $\kappa_\ell=\am$, and consider the choice
\begin{eqnarray}
\label{def:A}
A & := & \left\{\xio_\ell = \{1,\ldots,\am\} \right\},
\end{eqnarray}
where we recall that $\xio_\ell$ is the $\ell-$colored region. Since $\xi_\ell$ is uniformly distributed over all $\am-$element subsets of $[n]$, we have
\begin{eqnarray*}
|A|& = & \frac{|\Omega_\kappa|}{{n\choose \am}}.
\end{eqnarray*}
On the other hand, from any state inside $A$, there are precisely $\am(n-\am)$ transpositions that result in a state outside $A$, and hence
\begin{eqnarray*}
|\partial A| & = & \am(n-\am)|A|.
\end{eqnarray*} 
Thus, the inequality (\ref{iota}) gives
\begin{eqnarray}
\label{sharp}
\iota(\kappa) & \le & \frac{\am(n-\am)}{\log{n \choose  \am}}\\
\nonumber & \le & \frac{n}{\log{\left(\frac{n}{\am}\right)}},
\end{eqnarray}
where the second line uses the classical binomial estimate ${n\choose k}\ge \left(\frac{n}{k}\right)^k$, valid for all $1\le k\le n$. This establishes the upper bound in Corollary \ref{co:exp}, as well as the lower bound in Theorem \ref{th:main}, by Lemma \ref{lm:vs}. Finally, note that in the case $\kappa=(\lfloor n/2\rfloor,\lfloor n/2\rfloor)$, the estimate (\ref{sharp}) yields
\begin{eqnarray}
\label{sharpiota}
 \frac{n}{\iota\left(\left\lfloor \frac n2\right\rfloor,\left\lfloor \frac n2\right\rfloor\right)} & \ge &   4\log\left(\frac{n}{\am}\right)-o(1).
\end{eqnarray}
Thus, our pre-factor can not be improved by more than $\log 2$, as claimed in Remark \ref{rk:sharp}. 
\end{proof}

\subsection{Coarsening argument}
\label{sec:colored}
It now remains to prove Lemma \ref{pr:insensitive} and Corollary \ref{co:colored}. Both will rely on the elementary observation, already alluded to in Remark \ref{rk:coarsening}, that  the multislice $\Omega_\kappa$ is a ``coarsened'' version of the ``free'' multislice $\Omega_{(1,\ldots,1)}$, where $(1,\ldots,1)$ denotes the all-one vector of length $n$. To formalize this, let us introduce the projection 
 $\Psi\colon [n]\to[L]$ defined by the relation
\begin{eqnarray*}
\Psi(i) = \ell & \Longleftrightarrow & i\in\left[\kappa_1+\cdots+\kappa_{\ell-1}+1,\kappa_1+\cdots+\kappa_{\ell}\right],
\end{eqnarray*}
and extend this definition to sequences by coordinate-wise application:
\begin{eqnarray*}
\Psi(\omega_1,\ldots,\omega_n) & := & \left(\Psi(\omega_1),\ldots,\Psi(\omega_n)\right).
\end{eqnarray*}
The mapping $\Psi$ ``projects'' the multislice $\Omega_{(1,\ldots,1)}$ onto $\Omega_\kappa$ in the following precise sense.
\begin{lemma}[Coarsening]\label{lm:coarsening}For any observable $f\colon\Omega_\kappa\to\R$, we have
\begin{eqnarray*}
\label{isomorphic}
\EE_\kappa\left[f\right] & = & \EE_{(1,\ldots,1)}\left[f\circ\Psi\right].
\end{eqnarray*}
Moreover, for any $f,g\colon\Omega_\kappa\to\R$ and any weighted graph $G=(G_{ij})_{1\le i,j\le n}$,
\begin{eqnarray*}
\label{coarsening}
\cE_\kappa^G\left(f,g\right) & = & \cE_{(1,\ldots,1)}^G\left(f\circ\Psi,g\circ\Psi\right).
\end{eqnarray*}
\end{lemma}
\begin{proof}
By construction, we have $|\Psi^{-1}(\{\ell\})|=\kappa_\ell$ for each color $\ell\in[L]$, and hence $\Psi$ maps $\Omega_{(1,\ldots,1)}$ to $\Omega_\kappa$. The first claim asserts that the $\Psi-$image  of the uniform measure on $\Omega_{(1,\ldots,1)}$ is the uniform measure on $\Omega_\kappa$, which is nothing more than the observation that each element of $\Omega_\kappa$ admits the same number of pre-images under $\Psi$ (namely $\kappa_1!\cdots\kappa_L!$). The second claim follows from the first and the definition (\ref{def:weighted}), once we note that the commutativity relation
\begin{eqnarray*}
\nabla^{ij}(f\circ\Psi)& = & (\nabla^{ij}f)\circ\Psi,
\end{eqnarray*}
trivially holds for all $1\le i<j\le n$ and all $f\colon\Omega_\kappa\to\R$. 
\end{proof}
We can now easily establish our log-Sobolev estimate for the colored exclusion process.
\begin{proof}[Proof of Corollary \ref{co:colored}]We use Lemma \ref{lm:coarsening} and the definitions of $\tls(\kappa)$ and $c(G)$ to write
\begin{eqnarray*}
\ent_\kappa(f) & \le & \tls(\kappa)\cE_\kappa(\sqrt{f},\sqrt{f})\\
& = & \tls(\kappa)\cE_{(1,\ldots,1)}(\sqrt{f}\circ\Psi,\sqrt{f}\circ\Psi)\\
& \le & \tls(\kappa)c(G)\cE_{(1,\ldots,1)}^G(\sqrt{f}\circ\Psi,\sqrt{f}\circ\Psi)\\
& = & \tls(\kappa)c(G)\cE_{\kappa}^G(\sqrt{f},\sqrt{f}).
\end{eqnarray*}
Since $f\colon\Omega_\kappa\to\R_+$ was arbitrary, we have just proved
\begin{eqnarray*}
\tls(\kappa,G) & \le & c(G)\tls(\kappa).
\end{eqnarray*}
The claimed upper bound now follows from our main estimate on $\tls(\kappa)$. The lower bound
\begin{eqnarray*}
\tls(\kappa,G) & \ge & 2\trel(G)
\end{eqnarray*}
is obtained by combining the general inequality $\tls(\cdot)\ge 2\trel(\cdot)$ with Theorem \ref{th:aldous}.  To prove the other lower bound, we choose the test function $f={\bf 1}_A$ in the definition of the log-Sobolev inequality, with $A$ as in (\ref{def:A}). We have already seen that $|A|=|\Omega_\kappa|/{n\choose \am}$. Moreover, we now have $|\partial A|\le |A|d\am$ where $d$ denotes the degree in $G$, since moving from $A$ to $A^c$ requires transposing some site in  $\{1,\ldots,\am\}$ with one of its $d$ neighbors. We thus obtain
\begin{eqnarray*}
\tls(\kappa,G) & \ge & \frac{|A|d\log\frac{|\Omega_\kappa|}{|A|}}{|\partial A|} \ \ge \  \frac{1}{\am}\log {n\choose\am} \ \ge \ \log\left({\frac n\am}\right),
  \end{eqnarray*}  
and the proof is complete.
\end{proof}
\begin{proof}[Proof of Lemma \ref{pr:insensitive}]The statement $\trel(\kappa)=1$ is a (simple) special case of Theorem \ref{th:aldous}. This immediately implies $\tmls(\kappa)\ge \frac 12$, by the general relation (\ref{order}). To prove the more interesting statement $\tmls(\kappa)\le 2$, we take an arbitrary function $f\colon\Omega_\kappa\to\R_+$ and use Lemma \ref{lm:coarsening} to write
\begin{eqnarray*}
\ent_\kappa(f) & = & \ent_{(1,\ldots,1)}(f\circ\Psi)\\
& \le & \tmls(1,\ldots,1)\cE_{(1,\ldots,1)}(f\circ\Psi,\log f\circ\Psi)\\
& = & \tmls(1,\ldots,1)\cE_{\kappa}(f,\log f).
\end{eqnarray*}
This shows that $\tmls(\kappa)\le\tmls(1,\ldots,1)$, and the desired conclusion now follows from  the  classical estimate $\tmls(1,\ldots,1)\le \frac 12$, due to Goel \cite[Corollary 3.1]{MR2094147}.
\end{proof}

\section*{Acknowledgment}The author warmly thanks Jonathan Hermon for his valuable comments on a preliminary version of this work. 
\bibliographystyle{plain}
\bibliography{draft}
\end{document}